\documentclass[11pt]{amsart}

\usepackage[all]{xy}
\usepackage[utf8]{inputenc}
\usepackage[english]{babel}

\usepackage{hyperref, cleveref}
\usepackage{amsmath,amssymb,commath}
\usepackage{enumerate,multicol}
\usepackage{tikz-cd}
\usepackage{xcolor}

\newtheorem{theorem}{Theorem}[section]
\newtheorem{proposition}[theorem]{Proposition}
\newtheorem{corollary}[theorem]{Corollary}
\newtheorem{lemma}[theorem]{Lemma}

\theoremstyle{definition}
\newtheorem{definition}[theorem]{Definition}
\newtheorem{example}[theorem]{Example}

\theoremstyle{remark}
\newtheorem{remark}[theorem]{Remark}

\numberwithin{equation}{section}

\DeclareMathOperator{\ga}{\textsl{g}}

\def\set#1{\left\{ #1 \right\}}

\def\rm#1{\mathrm{#1}}
\def\cal#1{\mathcal{#1}}

\usepackage[all]{xy}

\tikzset{
  symbol/.style={
    draw=none,
    every to/.append style={
      edge node={node [sloped, allow upside down, auto=false]{$#1$}}}
  }
}

\UseRawInputEncoding

\begin{document}

\title[SRF and the prescribing scalar curvature problem]{Singular Riemannian Foliations and the prescribing scalar curvature problem}

	\author[Alexandrino]{Marcos M. Alexandrino}
	\address{ R. do Mat\~{a}o, 1010 - Butant\~{a}, S\~{a}o Paulo - SP, 05508-090.}
	\email{malex@ime.usp.br}
	
		\author[Cavenaghi]{Leonardo F. Cavenaghi}
		\address{University of Fribourg,
Departement of Mathematics,
Ch. du mus\' ee 23, CH-1700 Fribourg.}
	\email{leonardofcavenaghi@gmail.com, leonardo.cavenaghi@unifr.ch}

	
	
	\begin{abstract}
	
	An orbit-like foliation is a singular foliation on a complete  Riemannian manifold $M$
 whose leaves are locally equidistant (i.e., a singular Riemannian foliation) and (transversely) infinitesimally homogenous. This class of singular foliation
contains not only the classe of partion of the space into orbits of isometric actions, but also infinite many non homogenous examples and
in particular the partition of $M$ into orbits of a proper groupoid. 
	
In this paper we prove  a version of Kondrakov Embedding Theorem and an analogous Principle of Symmetric Criticality of Palais
for basic funcions of  orbit-like foliations. 	
As proof of concepts, we study not only the corresponding Yamabe problem in the setting, but also to the case of fiber bundles with homogeneous fibers, seeking for the existence of metrics with constant scalar curvature that respect the respective Riemannian Foliation decomposition. An application to the existence of positive constant scalar curvature on exotic spheres is presented. In an upcoming version we shall extend the results to the corresponding Kazdan--Warner problem.

	\end{abstract}
\maketitle



\section{Introduction}




It is classical nowadays (see for instance \cite{hebey1,hebey2,Hebey_2000})) that geometric analytic problems modeled on manifolds equipped with symmetries coming from group actions are easier to deal given both, the existence of better compactness embeddings of Sobolev spaces in Lebesgue spaces \cite{hebey3}, and the classical Principle of Symmetric Criticality due to Palais \cite{palais1979}.

A recent proof of concept of the aforementioned discussion is presented in \cite{cavenaghi2021}, where the authors pose and solve the analogous Kazdan--Warner problem (\cite{kazdaninventiones,kazadanannals,kazdan1975}) in the setting of Riemannian manifolds with isometric group actions: \emph{which invariant functions are the scalar curvature of Riemannian metrics?}


On the other hand, a partition of the Riemannian manifolds 
into  orbits of isometric actions are particular examples of \emph{singular Riemannian
foliations} (SRF for short), i.e., singular foliations
 locally equidistant (see Definition \ref{MALEX-definition-SRF}). This more general class of singular foliations,
appear naturally in different context.  
In particular, the partition of orbits of a proper groupoid (recall Definition \ref{MALEXdefinition-Lie-groupoids}
 and  discussion  in Section \ref{MALEXsec-Sasaki metrics-SRF-holonomy-groupoid}) 
are examples of the so called \emph{orbit-like foliation}, i.e.,
 SRF whose restriction to each slice is homogenous, i.e., coming from an action of a compact group (see Definition \ref{MALEX-definition-orbit-like}). 
As far as we know, even the  most classical examples of non homogenous orbit-like foliations of codimention one in spheres (see  \cite{FKM})
and the more recent examples of orbit-like foliations of  codimension greater than one in spheres  (see \cite{radeschi2014clifford}) 
may not come from proper (global)  groupoids.
Thus, at least until to this date, orbit-like foliations are a broader class of singular foliations than those whose leaves are orbits of proper grupoids.


Pursuing to establish an analytic framework to consider both physical and geometric problems modeled on Riemannian manifolds with singular Riemannian foliations, 
and in particular orbit-like foliations, in this paper we synthesize 
and disseminate the concept of the \emph{Sobolev space of basic distributions}, which consists in the Banach space of distributions that are constant along the leaves of a Singular Riemannian Foliation. In particular we prove  a Kondrakov type result. 

\begin{theorem}[Kondrakov-type theorem]\label{thm:kondrakovintro}
Suppose that $M$ is a connected compact Riemannian manifold and $\mathcal{F}$ is a SRF on $M$ whose leaves are closed.
Then there exists $p_0 > p^* := np/(n-p)$ such that given $1<q<p_0$, the canonical embedding $W^{1,p}(M)^\mathcal{F} \hookrightarrow L^q(M)$ is compact.
\end{theorem}
The well definition of Sobolev space of basic distribution  is possible given the existence of a natural \emph{basic projection operator}, which plays the analogous role of the classical \emph{average operator} in Riemannian manifolds with group actions (see \cite{brebook}). This operator
also allow us to prove \emph{principle of Basic criticality for orbit-like foliations}, see Lemma \ref{MALEX-lemma-palais-l-orbit-like-version}. 

Since both Yamabe (\cite{trudinger,aubin,schoen-yau1,schoen-yau2,schoen-yau3}) and Kazdan--Warner problems have re-gaining huge interest 
one naturally uses the developed machinery to approach these in our scenario. Namely, we prove:

	\begin{theorem}\label{thm:orbitlikeintro}
		Let $ M^n,~n\geq 3$, be a closed Riemannian manifold endowed with an orbit-like foliation $\cal F$ with closed leaves. 
		Then $ M$ has a Riemannian metric of constant scalar curvature for which $\cal F$ is an orbit-like foliation.
	\end{theorem}
	
	Once Riemannian submersions are a particular kind of manifold with Riemannian foliations, and metrics with curvature properties gain an additional constraint in this scenario given the isometry condition between horizontal space and the base Riemannian manifold, to stretch the range of the developed machinery, we also present a result concerning the existence of Riemannian submersion metrics of constant positive scalar curvature on fiber bundles with homogeneous fibers:

	\begin{theorem}\label{thm:bundlesintro}
		Let $ M^n,~n\geq 3$, be a closed Riemannian manifold endowed with a Foliation $\cal F$ induced by a fiber bundle such that: 
		\begin{enumerate}[$(i)$]
		    \item The structure group $G$ is compact and has non-abelian Lie algebra;
		    \item The fiber $L$ is an homogeneous space.
		\end{enumerate}
		 Then $M$ has a Riemannian metric of positive constant scalar curvature for which $\cal F$ is Riemannian.
	\end{theorem}
	
	Then a simple combination of the classical Eells--Kuiper invariant (\cite{eells}), which determine the number of diffeomorphism classes of exotic spheres that can be realized as the total sphere bundles; with Theorem \ref{thm:bundlesintro}, allow us to obtain the following:
	
	\begin{corollary}\label{thm:grourigasintro}
 16 (resp. $4.096$) from the $28$ (resp. $16.256$) diffeomorphisms classes of the 
$7$-dimensional (resp $.15$)-exotic spheres admit metrics of positive constant scalar curvature. Moreover, these can be taken as Riemannian submersion metrics when such spaces are considered as the total space of sphere bundles.
\end{corollary}



This paper is organized as follows. 
In Section \ref{MALEX-section-Linear-Lie-groupoid} we review the linear Lie groupoid associated to the semi-local description
of an orbit-like foliation $\mathcal{F}$ near a closed leave $B$ of $\mathcal{F}$. We try to present this in a  self contained presentation,  
 hopping to make it accessible also to readers without previous training in groupoid or singular Riemannian foliations.
Then a \emph{type of principle of symmetric criticality of Palais} is proved in a neighborhood of a leaf $B\in \mathcal{F}$;
see Lemma \ref{MALEX-lemma-palais-l-orbit-like-version}. In Section \ref{MALEX-section-principle-symmetry-palais-M}
we check that the  operator $J$ associated to the Yamabe problem  satisfies  the basic criticality principle, once
one considers a special basic metric. Then in Section
\ref{MALEX-sec-sobolev-spaces} Sobolev spaces of basic functions of SRF $\mathcal{F}$  and the $\mathcal{F}$-avarage operator on $M$ are presented and Theorem \ref{thm:kondrakovintro}
is proved. Finally  Theorems  
\ref{thm:orbitlikeintro}  and \ref{thm:bundlesintro}
as well Corollay as \ref{thm:grourigasintro} are proved  in Section \ref{MALEX-section-yamabe-problem}. 

\section*{Acknowledgments}
The authors are thankful to Jo\~ao Marcos do \'{O},  Llohann Speran\c{c}a  and Gustavo P. Ramos for fruitful conversations.

The first named author was supported by grant \#2016/23746-6, S\~{a}o Paulo Research Foundation (FAPESP).
The second named author was supported by the SNSF-Project 200020E\_193062 and the DFG-Priority programme SPP 2026.


\section{Linear Lie groupoid and criticality }
\label{MALEX-section-Linear-Lie-groupoid}
\subsection{A few facts about singular Riemannian foliations, orbit-like foliations and the Linear groupoid}
\label{MALEX-subsection-Linear-afewfacts}


In this section we review serveral facts on singular Riemannian foliations $\mathcal{F}$, 
most of them can be found in \cite{alexandrino_radeschi_2017,alexandrino2021lie}.
We also stress that, along this paper, the leaves of $\mathcal{F}$ are closed on a compact manifold $M.$

\subsubsection{Singular Riemannian foliations}
\label{MALEX-Section-SRF}

\begin{definition}[SRF]
\label{MALEX-definition-SRF}
A \emph{singular Riemannian foliation} on a complete Riemannian manifold $M$ is a partition $\mathcal{F}=\{ L\}$ of $M$ into immersed 
submanifolds  without self-intersections (the \emph{leaves}) that satisfies the following properties:
\begin{enumerate}
\item[(a)] $\mathcal{F}$ is a \emph{singular foliation}, i.e., for each $v_p$ tangent to $L_p$ (i.e., the leaf through $p\in M$) then there exists a local vector field $\vec{X}$ so that $\vec{X}(p)=v_p$ and $\vec{X}$ is tangent to the leaves; 
\item[(b)] $\mathcal{F}$ is  \emph{Riemannian}, i.e, each geodesic $\gamma$ that starts orthogonal to a leaf $L_{\gamma(0)}$ remains
orthogonal to all leaves that it meets.
\end{enumerate}
\end{definition}
\begin{remark}
Item (a) is equivalent to saying  that given a point $q\in M$, 
there exists a neighborhood $U$ of $q$ in $M$, a simple foliation  $\mathcal{P}=\{P\}$ on $U$ (i.e, given by fibers of a submersion on $U$)
so that the leaf  $P_{q}\in \mathcal{P}$ (the plaque through $q$) is a  relative compact open set of the leaf $L_q$  and 
$\mathcal{P}$ is a subfoliation of $\mathcal{F}|_{U},$ i.e, 
for each $x\in U$ we have $P_{x}\subset L_{x}$. From now on $\mathcal{P}\subset \mathcal{F}|_{U}$ denotes to be a subfoliation.  
In particular item (a)
implies that $\mathcal{F}\cap S_{q}$ is a singular foliation for each transverse submanifold $S_q$, i.e., $T_{q}M=T_{q}S_{q}\oplus T_{q}L_{q}$. 
Roughly speaking item (b) says that the leaves are \emph{locally equidistant}. In other words, 
item (b) is equivalent to saying that there exists $\epsilon>0$ so that if $x\in \partial \mathrm{tub}_{\epsilon}(P_q)$ (the cilinder of radius $\epsilon$ of the 
plaque $P_q$)  then the connected component of $L_x\cap U$ containing $x$  is contained in $\partial \mathrm{tub}_{\epsilon}(P_q).$ 
\end{remark}

Typical examples of singular Riemannian foliations (SRF for short) are, among others,  
the partition of $M$ into orbits of isometric actions; 
infinite many examples of nonhomogenous SRF on Euclidean spheres constructed by Radeschi using Clifford system \cite{radeschi2014clifford};
the holonomy foliation in a Euclidean fiber bundle with a connection compatible with the metric of the fibers (see Example 
\ref{MALEX-ex-transformation-holonomy-groupoid}).

Several properties of SRF are natural generalizations of  classical properties  of the partition of $M$ into orbits of isometric actions, see \cite{alexandrino2015lie}.
Le us review a few of them. 

The first one is the generalization of  the so called   \emph{slice representation}. 
Let $\pi:U\to L_q$ be the
metric projection, and $S_q=\pi^{-1}(q)$ be the
\emph{slice} i.e., $S_{q}:=\exp_{q}(\nu_{q}L\cap B_{\epsilon}(0))$ where $\nu_{q}(L_q)$ is the normal space.  
Then the \emph{infinitesimal foliation}  $\mathcal{F}_{q}=\exp_{q}^{-1}\big(S_{q}\cap \mathcal{F}\big)$ 
turns to be a SRF on the open set of the  Euclidean space
$(\nu_{q}L_{q},g_q).$ 
The infinitesimal foliation  $\mathcal{F}_{q}$ on a neighborhood of  $\nu_{q}(L_q)$  can be extended  via the homothetic transformation 
$h^{0}_{\lambda}(v)=\lambda v$ to a SRF on $(\nu_{q}(L_q), g_q)$. 
The foliation $\mathcal{F}_{q}$ plays a role in the theory of SRF similar to the role played by the slice representation  in the theory of isometric  actions.


Another general property of SRF that is analogous to the theory of isometric action, is that 
\emph{the partition of $M$ into the leaves of $\mathcal{F}$
with the same dimension is a stratification.} Recall that 
a \emph{stratification} of $M$ is a partition of
$M$ into embedded submanifolds $\{M_{i} \}_{i\in I}$ 
(called strata) such that: 
\begin{enumerate}
\item[(i)] the partition is locally finite, i.e.,
each compact subset of $M$ only intersects a finite
number of strata; 
\item[(ii)] for each $i\in I,$ there exists a subset
$I_{i}\subset I/\{i\}$ such that the closure of $M_i$
is $\overline{M}_{i}= M_{i}\cup \, \bigcup_{j\in I_{i}} M_{j}$;
\item[(iii)] $\dim M_{j}< \dim M_{i}$ for all $j\in I_{i}$
\end{enumerate}
The stratum with the leaves of  greatest dimension (the \emph{regular leaves}) \emph{is  
a open dense set, and its space of leaves  is connected.}

By successively blowing up along singular stratum we have a desingularization of SRF. More more precisely:
\begin{theorem}[\cite{Alexandrino-desingularization}]
\label{MALEX-theorem-desingularizacao}
Let $\mathcal{F}$ be a singular Riemannian foliations on a Riemannian
 compact manifold $(M,g)$ with compact leaves. 
For each small $\epsilon>0$ there exists 
a singular Riemannian foliation $\mathcal{F}_{\epsilon}=\{(L_{\epsilon})_{x} \}_{x\in M_{\epsilon}}$ 
on a Riemannian  compact manifold 
$(M_{\epsilon}, g_{\epsilon})$ and 
map $\pi_{\epsilon}: M_{\epsilon}\to M$ so that:
\begin{enumerate}
\item[(a)] $\pi_{\epsilon}$ projects leaves of 
$\mathcal{F}_{\epsilon}$ to leaves of $\mathcal{F}$; 
\item[(b)] let $\Sigma$ be the singular stratum and
$\Sigma_{\epsilon}=\pi_{\epsilon}^{-1}(\Sigma)$   
then 
$\pi_{\epsilon}:M_{\epsilon}-\Sigma_{\epsilon}\to
M-\Sigma$ is a foliated diffeomorphism;
\item[(c)] $|d\big(L_{\pi_{\epsilon}(x)}, L_{\pi_{\epsilon}(y)}\big)- 
d_{\epsilon}\big((L_{\epsilon})_{x}, (L_{\epsilon})_{y}\big) 
|<\epsilon $.
 \end{enumerate}  
In particular the metric space $M/\mathcal{F}$ is a Gromov-Hausdorf limit of a sequence of Riemannian
orbifolds $\{M_{n}/\mathcal{F}_{n}\}.$ 
\end{theorem}

In the next section, we will consider a particular type of SRF (the so called orbit-like foliation) that is fundamental
to understand the semi-local model of SRF; see Theorem \ref{MALEX-theorem-modelo-semi-local}

\subsubsection{Linearized foliations and orbit-like foliations }

Given a closed leaf $B=L_q$ we can always find a  
$\mathcal{F}$-saturated tubular neighborhood  $U=\mathrm{tub}_{\epsilon}(L_{q})$ of $L_q.$
The foliation restricted to $U$, i.e., $\mathcal{F}|_{U}$ (and in particular the partition by plaques) 
are invariant by the homothetic transformation 
$h_{\lambda}: U\to U$ defined as $h_{\lambda}(\exp(v))=\exp(\lambda v)$ for each $v\in\nu^{\epsilon}(B)$
where $\lambda\in(0,1]$. 

For each smooth vector field $\vec{X}$ in $U$ tangent to $\mathcal{F}$, we associate
a smooth vector field $\vec{X}^{\ell}$, called the \emph{linearization of
$\vec{X}$} with respect to $B$ as:

$$\vec{X}^{\ell}(q)=\lim_{\lambda\to 0} (h_{\lambda}^{-1})_{*} (\vec{X})\circ h_{\lambda}(q)$$
Since the restricted foliation $\mathcal{F}|_{U}$ is homothetic invariant, $\vec{X}^{\ell}$ is still tangent to $\mathcal{F}.$
It is possible to prove  that 

\begin{lemma}
\label{MALEX-lemma-fluxo-isometiras} 
The flows of these vector fields, once
identified with the normal exponential map, induce  isometries
on the fibers of the normal $\delta$-fiber bundle $E^{\delta}=\nu^{\delta}(B)=\{\xi\in\nu(B), \|\xi\|<\delta \}.$

\end{lemma}

\begin{example}
\label{MALEX-example-linearized-vector-field}
Given a SRF $\mathcal{F}$ with compacts leaves  on $\mathbb{R}^{m}$, and $B=0$
we have for $\vec{X}$ tangent to the leaves of $\mathcal{F}$
that 
$\vec{X}^{\ell}(v) =
\lim_{\lambda\to 0}\frac{1}{\lambda}\vec{X}(\lambda v)
=(\nabla_{v}\vec{X})_{0}$
i.e, $\vec{X}^{\ell}$ is in fact a linear vector field. In addition, 
one can check that it is also a Killing vector field. This can be proved
using the fact that the leaves are tangent to the spheres and hence
$0=\langle \vec{X}^{\ell}(v),v\rangle= \langle (\nabla_{v}\vec{X})_{0},v\rangle.$

Note that  the Killing vector fields $\vec{X}^{\ell}$ induce a Lie algebra
of a  connected Lie subgroup $K^{0}\subset O(n)$. Since by hypothesis $\mathcal{F}$ is compact, it is possible 
to check that $K^{0}$ is also compact. We have then in this example an homogenous subfoliation 
$\mathcal{F}^{\ell}=\{K^{0}(v)\}_{v\in\mathbb{R}^{m}}\subset \mathcal{F}$.
This turns to be the maximal homogenous subfoliation of $\mathcal{F}.$
\end{example}


The above example illustrates a  more general phenomenon.

\begin{definition} 
\label{MALEX-definition-linearized-foliation}
Given a SRF $\mathcal{F}|_{U}$ with compact leaves, 
the composition of linearized flows tangent to  $\mathcal{F}|_{U}$
induces a singular subfoliation  $\mathcal{F}^{\ell}\subset \mathcal{F}|_{U}$ on $U$, the so called 
\emph{linearized foliation}. 
\end{definition}

$\mathcal{F}^{\ell}$  can also been seen as  
the maximal infinitesimal homogenous subfoliation
of  $\mathcal{F}|_{U}.$ 
In other words, Let $\pi:U\to L_q$ be the
metric projection, and $S_{q}=\pi^{-1}(q)$ be a
slice, i.e., $S_{q}:=\exp_{q}(\nu_{q}L\cap B_{\epsilon}(0))$. 
Define  $\mathcal{F}_{q}^{\ell}$ as the   extension  
of $\exp^{-1}_{q}\big(S_{q}\cap \mathcal{F}^{\ell}\big)$ via the homothetic transformation $h_{\lambda}^{0}(v)=\lambda v$.
The foliation $\mathcal{F}_{q}^{\ell}$
is the maximal homogenous subfoliation of 
the infinitesimal foliation $\mathcal{F}_{q}.$

\begin{definition}[Orbit-like foliation]
\label{MALEX-definition-orbit-like}
A SRF  $(M,\mathcal{F})$ with compact leaves is called 
 \emph{orbit-like} if for each leaf $B=L_{q}$ we have $\mathcal{F}^{\ell}=\mathcal{F}|_{U}$. In other words
if for each $q\in M$ the infinitesimal foliation $\mathcal{F}_{q}$ is homogenous, and the leaves
are orbits of a compact (isometric) group $K_q^{0}$. 
\end{definition}

\begin{remark}
\label{MALEX-remark-orbit-like-topological}
To be orbit-like  could be consider a topological property in the following sence:
Let $(M_i,\mathcal{F}_{i})$ be two SRF and $\psi:(M_{1},\mathcal{F}_{1})\to (M_{2},\mathcal{F}_{2})$ be a foliated diffeomorphism. Then
$(M_{2},\mathcal{F}_{2})$ is orbit-like if and only if $(M_{1},\mathcal{F}_{1})$ is 
 orbit-like, see \cite{alexandrino_radeschi_2017}.

\end{remark}

\subsubsection{Sasaki metrics, SRF and holonomy groupoid}
\label{MALEXsec-Sasaki metrics-SRF-holonomy-groupoid}
Let $E=\nu(B)$ be the normal bundle of a compact leaf  $B=L$ of  
 $\mathcal{F}.$  Consider the Euclidean vector bundle $\mathbb{R}^{k}\to E\to B$
where the metric on each fiber $E_p$ 
is defined as the metric $g_p$ for $p\in B$. 

By pulling back via the normal exponential map, we can 
identify the foliation $\mathcal{F}^{\ell}$ and $\mathcal{F}|_{U}$
to singular foliations on open set $(\exp^{\nu})^{-1}(U)$ of $E$,
and from now
on we use the same notation for the foliations on $U$ or on 
$(\exp^{\nu})^{-1}(U).$
By homothetic transformation we can extend 
$\mathcal{F}^{\ell}$ and $\mathcal{F}$ to $E$.

We recall  that there exists a Sasaki
metric $\mathrm{g}^{0}$ on $E$ so that $\mathcal{F}^{\ell}$ and $\mathcal{F}$ are singular Riemannians foliations.
In fact, we can find a  distribution  $\mathcal{T}$ homothetic invariant
that is tangent to $\mathcal{F}^{\ell}$ and $\mathcal{F}$. 
This distribution can be constructed by finding a distribution $\widehat{\mathcal{T}}$ tangent 
to $\mathcal{F}$ and then linearizing the vector fields tangent to $\widehat{\mathcal{T}}$. In particular 
there may exist different  
ways to construct $\mathcal{T}$.  
 Let $\mathrm{g}^{0}$  be the \emph{associated Sasaki metric}, i.e., 
the metric so  that:
\begin{itemize}
\item $\mathcal{T}$ is orthogonal to the fibers $E$,  
\item the foot point projection,   
 $\pi:(E,\mathrm{g}^{0})\to (B,g)$ is a Riemannian submersion 
\item and the fibers $E_{p}$ have the flat metric $g_{p}$. 
\end{itemize}  
	
	Let us denote $\nabla^{\tau}$ the connection associated to the  distribution $\mathcal{T}.$  It is possible to check that  
	$\nabla^{\tau}$  is compatible with the Euclidean metric on the fibers of $E$. 

As we are going to recall in  Theorem \ref{MALEX-theorem-modelo-semi-local}, 
 we can use  the connection $\nabla^{\tau}$ 
 to describe  $\mathcal{F}$ and its linearization
$\mathcal{F}^{\ell}$. 
In order  to better understand Theorem \ref{MALEX-theorem-modelo-semi-local}, we need  to  recall  the definition of Lie groupoid. 
\begin{definition}
\label{MALEXdefinition-Lie-groupoids}
A \emph{Lie groupoid} $\mathcal{G}=\mathcal{G}_{1}\rightrightarrows \mathcal{G}_{0}$ consist of:
\begin{enumerate}
\item   a manifold $\mathcal{G}_{0}$ called the \emph{set of objects};
\item    a  (possible non Hausdorff) manifold
 $\mathcal{G}_{1}$ called   \emph{the set of arrows}  (between objects); 
\item  submersions $s,t:\mathcal{G}_{1}\to\mathcal{G}_{0}$
which associate to an arrow $g\in\mathcal{G}_{1}$ 
its \emph{source} (i.,e $s(g)$) and its \emph{target} 
(i.e, $t(g)$) respectively;
\item  a multiplication map $m: \mathcal{G}_{2}\to\mathcal{G}_{1},$ $m(g,h)=gh$ where 
$\mathcal{G}_{2}=\{(g,h)\in\mathcal{G}_{1}\times\mathcal{G}_{1} \, | s(g)=t(h) \},$ that satisfies $s(g h )=s(h)$ and $t(g h)=t(g)$;
\item  a global section $\textbf{1}:\mathcal{G}_{0}\to\mathcal{G}_{1}$ called \emph{unit} that satisfies  
$t(\textbf{1}(x))=x=s(\textbf{1}(x))$, $\mathbf{1}_{x}h=h$ and $g\mathbf{1}_{x}=g$ for all $h\in t^{-1}(x)$ and $g\in s^{-1}(x)$; 
\item  a diffeomorphism  $i:\mathcal{G}_{1}\to\mathcal{G}_{1}$, $i(g)=g^{-1}$ called \emph{inverse map} that satisfies $s(g^{-1})=t(g)$, $t(g^{-1})=s(g),$ $g g^{-1}=\textbf{1}_{t(g)},$
$g^{-1}g=\textbf{1}_{s(g)}.$ 
\end{enumerate}
\end{definition}
A Lie groupoid  $\mathcal{G}=\mathcal{G}_{1}\rightrightarrows \mathcal{G}_{0}$    
induces a singular foliation
on $\mathcal{G}_{0}$ whose leaves are the connected
components of the \emph{orbits} of $\mathcal{G}$, i.e.,
$\mathcal{G}(x)=\{t (s^{-1}(x) ) \}.$
A Lie groupoid $\mathcal{G}=\mathcal{G}_{1}\rightrightarrows \mathcal{G}_{0}$ is called \emph{proper Lie groupoid} if 
the map $\psi:\mathcal{G}_{1}\times\mathcal{G}_{1}\to\mathcal{G}_{0}\times \mathcal{G}_{0}$ defined as $\psi(g)=(s(g),t(g))$ is proper. 

\begin{example}[Holonomy groupoid]
\label{MALEX-ex-holonomy-groupoid}
Consider an Euclidean bundle $\mathbb{R}^{n}\to E\to B$ with a compatible connection $\nabla^{\tau}.$ 
Given a piece-wise smooth curve $\alpha:[0,1]\to B$, let $\|_{\alpha}$ be  the $\nabla^{\tau}$-parallel transport along $\alpha$.
In this case:
\begin{enumerate}
\item $\mathcal{G}_{0}=B$; 
\item $\mathcal{G}_{1}=\{\|_{\alpha}, \,  \, \forall \alpha:[0,1]\to B\, \, \mathrm{piece wise \, smooth} \}$;
\item the source is $s(\|_{\alpha})$=$\alpha(0)$, the target is $t(\|_{\alpha})=\alpha(1);$
\item the multiplication is  produced with concatenation, i.e., $m\big(\varphi_{\beta},\varphi_{\alpha}\big)=\varphi_{\beta*\alpha}$;
\item the unit $\textbf{1}_{\alpha(0)}=Id$;
\item the inverse: $(\|_{\alpha})^{-1}=\|_{\alpha^{-1}}$. 
\end{enumerate}
  \end{example}
The above Lie groupoid is not appropriate to describe  the parallel transport of vectors of $E$, since
$\mathcal{G}_0=B.$ In order to correct this problem, we need to construct a new Lie groupoid so that the set of objects  is $E$. 
At the same time, we briefly review (in a concrete example)  how to construct a transformation Lie groupoid of a representation  
$\mu: \mathcal{G}_{1}\times_{\mathcal{G}_0} E\to E$  of a Lie groupoid $\mathcal{G}_{1}\rightrightarrows \mathcal{G}_{0}.$ 

\begin{example}[Transformation holonomy groupoid]
\label{MALEX-ex-transformation-holonomy-groupoid}
Consider the notation of  Example  \ref{MALEX-ex-holonomy-groupoid}. 
Set $\mathcal{G}_{1}\times_{B} E=\{(g,v_x)\in \mathcal{G}_{1}\times E \, | s(g)=x=\pi(v_x) \}.$ We define 
\emph{the representation} $\mu: \mathcal{G}_{1}\times_{B} E\to E$ as  
$\mu\big(\|_{\alpha}, v_{\alpha(0)}\big)=\|_{\alpha}v_{\alpha(0)}$ and the transformation
holonomy groupoid  $\mathrm{Hol}^{\tau}=\mathcal{G}\ltimes E$  as:
\begin{enumerate}
\item $(\mathcal{G}\ltimes E)_{0}=E$;
\item $(\mathcal{G}\ltimes E)_{1}=\{(\|_{\alpha}, v_{\alpha(0)})\in \mathcal{G}_{1}\times_{B} E  \}$;
\item for $g=\big(\|_{\alpha},v_{\alpha(0)}\big)$ we have $s(g)=v_{\alpha(0)}$ and $t(g)=\mu\big(\|_{\alpha},v_{\alpha(0)}\big)$;
\item  for $w=\|_{\alpha}v$, we set the multiplication as  $m\big( (\|_{\beta},w),(\|_{\alpha},v)\big)=(\|_{\beta*\alpha},v)$;
\item the unit $\textbf{1}(v)=(\textbf{1}_{\pi(v)},v)$;
\item the inverse $\big(\|_{\alpha},v_{\alpha(0)}\big)^{-1}=\big(\|_{\alpha^{-1}},\|_{\alpha}v_{\alpha(0)}\big)$.
\end{enumerate}
The orbits of $\mathrm{Hol}^{\tau}$ is a singular Riemannian foliation $\mathcal{F}^{\tau}$ on $E$  with respect to the Sasaki metric. 
\end{example}


\begin{theorem}[\cite{alexandrino2021lie}]
\label{MALEX-theorem-modelo-semi-local}
Consider the foliations $\mathcal{F}$ and $\mathcal{F}^{\ell}$ on the normal bundle $E=\nu(B)$ of a closed leaf $B$.
Let $\nabla^{\tau}$ be a Sasaki connection compatible with these foliations (i.e., so that the induced distribution $\mathcal{T}$ is tangent to them). 
Let  $K_{p}^{0}$ be the connected Lie group 
whose Lie algebra is associated to   the Killing vector fields of the linearization of infinitesimal foliation $\mathcal{F}_{p}$ on $E_{p}.$ 
   Then:
\begin{enumerate}
\item[(a)] $\mathcal{F}=\mathrm{Hol}^{\tau}(\mathcal{F}_{p})$
\item[(b)] $\mathcal{F}^{\ell}=\mathrm{Hol}^{\tau}(\{K_{p}^{0}(v) \}_{v\in E_{p}})$
\end{enumerate}

\end{theorem}

The above result already suggest that the leaves of 
$\mathcal{F}^{\ell}$ may coincide with the orbits of a (Lie) groupoid
 $\mathcal{G}^{\ell}$.  


\subsubsection{Linear Lie groupoid }
\label{MALEX-sec-Linear Lie groupoid }

Let us  recall how to construct  the linear groupoid 
$\mathcal{G}^{\ell}\rightrightarrows E=\nu(B)$ whose orbits
 are  leaves of the orbit-like foliation $\mathcal{F}^{\ell}$. 

We start by considering 
$O(E)$ the orthogonal fame bundle of $E=\nu(B)$. We recall that 
the leaves of the foliation $\mathcal{F}^{\ell}$ 
are orbits of flows of linearized vector fields, that induce   isometries between
the fibers on $E$ (see Lemma \ref{MALEX-lemma-fluxo-isometiras}),  and hence  can be lifted to flows on $O(E)$.  The 
orbits of these flows on the frame bundle are leaves of a regular foliation
$\widetilde{\mathcal{F}}$ on  $O(E).$

Note that the action of each isotropic group $K_{p}^{0}$
induces a free action on $O(E_{p})$ and the orbits 
of this action   are tangent to the leaves of 
$\widetilde{\mathcal{F}}$. 
 Also the connection $\nabla^{\tau}$ induces a linear
distribution $\widetilde{\mathcal{T}}$ tangent to the leaf 
of $\widetilde{\mathcal{F}}$. 
The tangent space of $\widetilde{\mathcal{F}}$ through $\xi\in O(E)$ 
can be described as 
$T_{\xi}\widetilde{L}_{\xi}=T K_{p}^{0}(\xi)\oplus 
\widetilde{\mathcal{T}}_{\xi}. $
We can induces a Riemannian metric $\tilde{g}$ on $T_{\xi}\widetilde{L}$ as follows: 
first define the  metric on $\widetilde{\mathcal{T}}$ so that $d\pi:\widetilde{\mathcal{T}}_{\xi}\to (TB,\mathrm{g})$
turns to be an isometry,  then we induces the metric on the orbits of  $K_{p}^{0}(\xi)$ using a bi-invariant metric.
It follows direct from this definition that the leaves of $\widetilde{\mathcal{F}}$ are locally isometric. 
By construction (and using the fact that $B$ is a leaf) one can check that the leaves of $\widetilde{\mathcal{F}}$ have trivial
holonomy and are in fact diffeomorphic (and hence isometric) to each other.  

 Set  
$\mathcal{G}:=\mathrm{Hol}(\widetilde{\mathcal{F}})/O(n)\rightrightarrows  O(E)/O(n)=B.$ 
Here $\mathrm{Hol}(\widetilde{\mathcal{F}})$ denotes  the holonomy grupoid of the foliation  $\widetilde{\mathcal{F}},$
that is defined as follows: the set of objects is the ambient space of the foliation $\widetilde{\mathcal{F}}$, 
an arrow  $\tilde{g}\in \big(\mathrm{Hol}(\widetilde{\mathcal{F}}) \big)_{1}$ 
is a class of a path in a leaf of $\widetilde{L}\in\widetilde{\mathcal{F}}$ joining $\tilde{s}(\tilde{g})\in \widetilde{L}$ 
with $\tilde{t}(\tilde{g})\in \widetilde{L}$, where the equivalence relation identifies  
paths inducing the same germ of diffeomorphisms 
sliding transversals along the paths; the multiplication, unit and  inverse  maps 
are defined by  concatenations, homotopy of a constant paths and  the inverse paths.



The Lie algebroid of $\mathcal{G}$ turns to be  $\mathcal{A}=T\widetilde{\mathcal{F}}/O(n)\to B$ and the metric $\tilde{\mathrm{g}}$ 
induces a Riemannian metric on $\mathcal{A}.$
Let us denote $\tilde{\nu}$ the density associated to the  Riemannian metric. 


Since $E=O(E)\times_{B}\mathbb{R}^{n},$ we can define the representation  
$\mu:\mathcal{G}_{1}\times_{\mathcal{G}_{0}} E\to E$ as $\mu(g,[\xi,e])=[\tilde{t}(\tilde{g}),e]$, where $\tilde{g}\in \mathrm{Hol}(\widetilde{\mathcal{F}})$ 
is the unique representative  of $g$ so that $\tilde{s}(\tilde{g})=\xi$. 
The linear Lie groupoid $\mathcal{G}^{\ell}$ is defined as the  transformation Lie groupoid of the representation $\mu$, i.e.,
$\mathcal{G}^{\ell}=\mathcal{G}\ltimes E.$  Recall that the target and source maps are  
$\mathrm{s}^{\ell}\big((g,v)\big)=v$ and $\mathrm{t}^{\ell}((g,v))=\mu(g,v).$ 
Note that $(\mathrm{s}^{\ell})^{-1}(v_x)=(\mathrm{s}^{-1}(x),v_x)$ where 
 $\mathrm{s}$  is the  source map of $\mathcal{G}_{1}\rightrightarrows B.$  
Let us sum up  the relation between the source fibers and leaves of $\widetilde{\mathcal{F}}.$  
\begin{lemma}
\label{lemma-facts-fiber-source-isometric}

\

\begin{enumerate}
\item The leaves of $\widetilde{\mathcal{F}}$ have trivial holonomy and are isometric to each other. 
\item The fiber of source map   $\mathrm{s}^{\ell}$ are isometric to (each)  leaf of $\widetilde{\mathcal{F}}.$
\end{enumerate}
\end{lemma}



\subsection{Linearized vector fields and volume}

\begin{lemma} 
\label{MALEX-flows-preserve-volume}
Given a Sasaki metric on $\mathbb{R}^{k}\to E\to B$ (possibly changing the metric on $B$ when $\dim B=1$), 
there exists a  module of  linearized vector fields  that  preserve  the volume $\nu$ (induced by the Sasaki metric). The composition of their
flows is transitive on the leaves of $\mathcal{F}^{\ell}.$  
 \end{lemma}
\begin{proof}

It suffices to construct  two types  of linearized vector fields   
that preserve the  volume $\nu$:   
\begin{enumerate} 
\item[Type 1:]  $\vec{F}\in\mathfrak{X}(E)$  whose orbits  restricts to $B$ are transitive on $B$, 
\item[Type 2:]  $\vec{G}\in\mathfrak{X}(E)$  whose orbits fix the fibers of $E$ and are transitive on the infinitesimal foliation on the fibers.   
\end{enumerate}
    
\emph{Constructing type 1 vector fields} 

First we consider a vector field $\vec{F}\in \mathfrak{X}(B)$ that preserves the volume $\nu_{B}$ of $B$. 
If $\dim B=0$ there is nothing to do. If $\dim B=1$, i.e, if $B$ is the circle $S^{1}$, it is easy to see the existence of a global  vector field that 
preserves the volume (changing the metric of $B$ if necessarily). So let us assume that $\dim B\geq 2$ and let us review the construction that given  a point $p\in B$ 
and small neighborhood $W\subset B$ of $p$  there exists a vector field $\vec{F}\in \mathfrak{X}(B)$ with a support on  neighborhood $W$ that preserves the volume. 
Consider a coordinate system so that $p$ is identified with $0$ and $\nu_{B}$ with $d x_{1}\wedge \cdots \wedge d x_{n}.$ 
Let $\widetilde{F}$ be an (Euclidean)  Killing vector field that fixes $0$ and 
a smooth  non negative function $\rho:\mathbb{R}\to\mathbb{R}$, with small  compact support in $(-\epsilon,\epsilon)$ and $\rho(0)=1$.  
Since the flow of $\widetilde{F}$ preserves $\nu_{B}$, then the flow of  $\vec{F}(x)=\rho(\|x\|) \widetilde{F}(x)$ 
also preserves $\nu_{B}$ and  has  compact support. By pulling back $\vec{F}$ via the coordinate system,  
 we identify the vector field $\vec{F}$ with a vector field on $B$, that we are also denoting as $\vec{F}.$

Now we can extend  the vector field $\vec{F}$ on $B$ to a $\pi$-basic vector field on $E$ so that
  $ \vec{F}\in\mathfrak{X}(\mathcal{T}).$ Since  
	$\mathcal{T}\subset T\mathcal{F}^{\ell}$  is homothetic, the vector field $\vec{F}$ is homothetic $\mathcal{F}^{\ell}$-vector field. 
Therefore $\vec{F}$ is linearized vector fields and hence it flows induces isometries between the fibers.

Let  $\nu_{E}$ be the volume form on the fibers $E$, note that  
$\pi^{*}\nu_{B}\wedge \nu_{E} $ coincides with the volume $\nu$ of the Sasaki metric, because 
$\pi:E \to B$ is a Riemannian submersion. 
\footnote{This can be easily  checked using an adapted  orthonormal frame $\{\tau_{i}\}$ of $\tau$ and $\{e_{\alpha}\}$ tangent to the fibers of $E$
and see that the Sasaki volume and this volume form coincides. }
Since $\pi\circ\varphi_{t}^{F}=\varphi_{t}^{F}\circ\pi$ and $\varphi_{t}^{F}$ preserves $\nu_{B}$ we infer that
$$(\varphi_{t}^{F})^{*}\Big(\pi^{*}\nu_{B}\wedge \nu_{E}\Big)=\pi^{*}\nu_{B}\wedge \nu_{E}$$

\emph{Constructing type 2 vector fields:}

Consider the connected isotropy group $K_{p}^{0}$. By parallel transport with respect to $\nabla^{\tau}$ we can induce an action of
$\mu:K_{p}^{0}\times\pi^{-1}(W)\to \pi^{-1}(W)$ on a neighborhood  $W\subset B$ of $p$  that fixes the fibers, and act on each fiber isometrically. 
For $\xi$ in the Lie algebra of $K_{p}^{0}$, set $\vec{G}(x)=d\mu_{x}\big(f(\pi(x)) \xi\big)$ for some smooth non negative function $f$ with compact support on $W$ and so that $f(p)=1.$  
Since its flow acts  isometrically on each fiber and its projection on $B$ is 
the identity we have that: 
$$(\varphi_{t}^{G})^{*}\Big(\pi^{*}\nu_{B}\wedge \nu_{E}\Big)=\pi^{*}\nu_{B}\wedge \nu_{E}$$

\end{proof}

\begin{remark}
\label{MALEX-remark-groupoid-volume-preserve}
The construction of the Lie groupoid 
presented in Section \ref{MALEX-sec-Linear Lie groupoid }
can be done using the flows of linearized vector fields
that preserve the volume  of the Sasaki metric.
\end{remark}


\subsection{Avarage operator of $\mathcal{F}^{\ell}$ }

We now define the avarage operator $\mathrm{Av}:C^{\infty}_{c}\big(E^{\delta}\big)\to C^{\infty}_{c}\big(E^{\delta}\big)_{b} $
that projects (as we will se below) smooth functions with support on $E^{\delta}$ 
(identified via exponential map with  $\mathrm{Tub}_{\delta}(B)$)
onto basic functions with support on  $E^{\delta}$.

\begin{equation}
\label{MALEX-definition-AV-prova-malex}
\mathrm{Av}(f)(v_x)=\frac{1}{\mathcal{V}(x)}\int_{(\mathrm{s}^{\ell})^{-1}(v_x)} f\circ \mathrm{t}^{\ell} \, \tilde{\nu}=
\frac{1}{\mathcal{V}}\int_{(\mathrm{s})^{-1}(x)} f\circ \mu(g,v_x) \, \tilde{\nu}_{x}
\end{equation}
where $\mathcal{V}(x)=\int_{(\mathrm{s})^{-1}(x)} \, \nu_{x}.$
Here we used the fact that $(\mathrm{s}^{\ell})^{-1}(v_x)=(\mathrm{s}^{-1}(x),v_x)$, where $\mathrm{s}$ is the source
map of $\mathcal{G}_{1}\rightrightarrows B.$

\begin{lemma}
\label{MALEX-lemma-Av-eh-basico}
Consider $f\in C^{\infty}_{c}(E^{\delta})$. Then $\mathcal{V}$ is constant and  $Av(f)$ is a basic function.  
\end{lemma}
\begin{proof}
The lemma  follows from  Lemma \ref{lemma-facts-fiber-source-isometric}.  
In fact, Lemma \ref{lemma-facts-fiber-source-isometric}  implies  directly that $\mathcal{V}$ is constant. 
It  also allows us to check  that  $Av(f)$ is a basic function,
analogously to how this fact is demonstrated in the classic case where the groupoid comes from the action of a group on $M.$ 
In other words,  set $x=\mathrm{s}(g_1)$ and   $y=\mathrm{t}(g_1)$  since the groupoid is transitive, i.e., 
$\mathrm{s}^{-1}(x)=\mathrm{s}^{-1}(y)=\widetilde{L}_{\xi}.$   
\begin{eqnarray*}
Av(f) \big(\mu(g_{1},v_x)\big) & = &\frac{1}{\mathcal{V} }\int_{\mathrm{s}^{-1}(y)} f\big(\mu(g_{2},\mu(g_{1},v_x))\big) \tilde{\nu}\\
& = & \frac{1}{\mathcal{V} } \int_{\widetilde{L}_{\xi}} f\big(\mu(g_{2}g_{1},v_x)\big) \tilde{\nu}\\
  & = &\frac{1}{\mathcal{V} } \int_{\mathrm{s}^{-1}(x)} f\big(\mu(g_{2}g_{1},v_x)\big) \tilde{\nu}
\end{eqnarray*}

\end{proof}

\label{sec:average}

\begin{definition} 
\label{MALEX-definition-Av-orbitlike}
A  linear functional  $l: C^{\infty}_{c}(E^{\delta}) \to \mathbb{R}$ is called \emph{symmetric} with 
with respect to the foliation $\mathcal{F}^{\ell}$ if it fullfils the following property: 
$l(f\circ \varphi^{\mathcal{F}})= l(f) $ for each $f\in C^{\infty}_{c}(E^{\delta})$ 
and for  each $\varphi^{\mathcal{F}}$ that is a composition of flows of  linearized vector fields 
that preserve the Sasaki metric.
\end{definition}

\begin{lemma}
\label{MALEX-lemma-Av-l-orbitlike} 
Let $l$ be a linear functional on $C^{\infty}_{c}\big(E^{\delta}\big)$  symmetric with respect to the foliation $\mathcal{F}^{\ell}.$ 
 Then for each  $f\in C^{\infty}_{c}\big(E^{\delta}\big)$
 $$ l\Big( \mathrm{Av}(f)  \Big)=l\Big( f \Big)$$



\end{lemma}
\begin{proof}

Consider a open cover $\{U_{\alpha}\}$ of $L$ such that $U_{\alpha}$ is difeomorphic to a $\dim L$-rectangle in Euclidean space, 
$\{\rho_{\alpha} \}$ the partion of unit subordinate to $\{U_{\alpha}\}.$ Set $f_{\alpha}(v_x)=f(v_x)\rho_{\alpha}(x)$. 
In order to prove the lemma it suffices to prove
\begin{equation}
\label{eq-0lemma-simetria-l}
l\Big(\int_{(\mathrm{s})^{-1}(x)} f_{\alpha}\circ \mu(g,v_x) \, \tilde{\nu}_{x} \Big)=
l\Big( f_{\alpha}\int_{(\mathrm{s})^{-1}(x)}  \, \tilde{\nu}_{x} \Big)
\end{equation}
Let  $\{\widetilde{U}_{i}\}$ be the connected component of  
$\{(\mathrm{t}|_{\mathrm{s}^{-1}(x)})^{-1}(U_{\alpha})\}.$ Note that $\widetilde{U}_{i}$ is a neighborhood of $\mathrm{s}^{-1}(x)$ that can be isometric identified
to a neighborhood of the leaf $\widetilde{L}_{\xi}$ of the foliation $\widetilde{\mathcal{F}}$ (used in the construction of the groupoid $\mathcal{G}_{1}$), i.e., 
$\widetilde{U}_{i}$ does not depend on $x$ and hence, from now, 
the neighboorhood $\widetilde{U}_i$ and its isometric neighborhood on  $\widetilde{L}_{\xi}$ are going to be denoted by the same notation.

\begin{equation}
\label{eq-1lemma-simetria-l}
\int_{(\mathrm{s})^{-1}(x)} f_{\alpha}\circ \mu(g,v_x) \, \tilde{\nu}_{x}  = \sum_{i} \int_{\widetilde{U}_{i}} f_{\alpha}\circ \mu(g,v_x) \, \tilde{\nu}_{i} 
\end{equation}
where $\tilde{\nu}_{i}=\rho_{i}\circ\mathrm{t}\,\tilde{\nu}_{x}.$
Also note that each neighboorhood $\widetilde{U}_i$ is diffeomorphic to $S\times G_{x_i}$ where $G_{x_i}$ is the connected component of the 
isotropic group of a point $x_{i}\in U_{i}$
and $S$ is Euclidean rectangle on the Euclidian space with dimension of $L$. Since $G_{x_{i}}$ is diffeomorphic to a connected compact group $K$ 
we have that $\widetilde{U}_{i}$ is diffeomorphic to $S\times K$. 
Let $\psi_{i}:S\times K\to \widetilde{U}_{i}$ be a parametrization, $f_{\alpha i}(\cdot,\cdot)=f_{\alpha}\circ \mu(\psi(\cdot), \cdot)$
and $\tilde{\nu}_{i}\omega_{0}=\psi_{i}^{*}\tilde{\nu}_{i}$. Then 
 writting in coordenates we have:  
\begin{equation}
\label{eq-2lemma-simetria-l}
\int_{\widetilde{U}_{i}} f_{\alpha}\circ \mu(g,v_x) \, \tilde{\nu}_{i}=\int_{S\times K} f_{\alpha i}(s,k,x,v) \sigma_{i}(s,k) \tilde{\nu}_0  
\end{equation}
Given an $\epsilon$ we claim that there exists a  partition $\{P_{i j}\}$ of $\tilde{U}_{i}$ so that
\begin{equation}
\label{eq-3lemma-simetria-l}
\int_{S\times K} f_{\alpha i}(s,k,x,v) \sigma_{i}(s,k) \tilde{\nu}_{0}=\sum_{j} f_{\alpha i}(s_{i j},k_{ ij},x,v) V_{i j} +R_{i}(x,v) 
\end{equation}
where $\|R_{i}\|_{W^{k,p}}<\epsilon$. Here  $V_{i j}= \sigma_{i}(s_{i j},k_{i j})\int_{\psi^{-1}(P_{i j})}\tilde{\nu}_{0}$ and $(s_{i j},k_{i j})\in \psi^{-1}(P_{i j}).$ In fact, define $h:\Big(S\times K \Big)\times E_{i}\to\mathbb{R}$  as 
$h\big((s,k), e\big):=f_{\alpha i}\big((s,k),e\big)\cdot\sigma_{i}(s,k)$ where $e=(x,v)\in E_i=\pi^{-1}(U_{i})\cap E^{\delta/2}$
Since $h$ is uniformly continous on $\Big(S\times K \Big)\times E_{i}$, we can check (by the definition of Riemann integral)
that $|R(x,v)|<\epsilon$ independent of $e=(x,v)\in E_{i}.$ 
Now if we set $\hat{h}=\frac{\partial^{\beta} }{\partial e_{\beta} }f_{\alpha i }\big((s,k),e\big)\cdot\sigma_{i}(s,k)$
and replace $\hat{h}$ in eq. \eqref{eq-3lemma-simetria-l} we infer, by similiar argument that $|\hat{R}(x,v)|<\epsilon_1$ 
once we consider a refinement of $\{P_{i j} \}.$ This fact and the partial derivation 
$\frac{\partial^{\beta} }{\partial e_{\beta} }$ of eq. \eqref{eq-3lemma-simetria-l} (using our original $h$)  imply that
$ |\frac{\partial^{\beta} }{\partial e_{\beta} } R(x,v)|<\epsilon.$ These facts 
allow us to conclude  that   $\|R_{i}\|_{W^{k,p}}<\epsilon$ for a refinement of $\{P_{i j} \}$.

Note that since the fibers of the source map are connected, each point $g\in \mathrm{s}^{-1}(x)$
is $g=\varphi^{\ell}(\mathrm{i}(x))$ for some $\varphi^{\ell},$ that is induced  
by composition of  flows of linearized vector field $\varphi^{\mathcal{F}}$ 
that preserve volume of the Sasaki metric (see Remark   \ref{MALEX-remark-groupoid-volume-preserve}). 
Hence we can rewrite $f_{\alpha}$ intrinsically as follows: 
\begin{equation}
\label{eq-4lemma-simetria-l}
f_{\alpha i}(s_{i j},k_{i j},x,v)=f_{\alpha}\circ\mu(\varphi_{i j}^{\ell}\mathrm{i}(x),v_{x})=
f_{\alpha}\circ\mathrm{t}^{\ell}\circ \varphi_{i j}^{\ell}\circ\mathrm{i}^{\ell}(v_{x})
\end{equation}

Given $\epsilon$, it follows from
eq. \eqref{eq-1lemma-simetria-l},\eqref{eq-2lemma-simetria-l},
\eqref{eq-3lemma-simetria-l},\eqref{eq-4lemma-simetria-l} that 
there is a  partition $\{P_{i j}\}$ of $\widetilde{U}_{i}$ so that:
\begin{equation}
\label{eq-5lemma-simetria-l}
\int_{\mathrm{s}^{-1}(x)} f_{\alpha}\circ \mu(g,v_x) \, \tilde{\nu}_{x}=
\sum_{i j} f_{\alpha}\circ\mathrm{t}^{\ell}\circ \varphi_{i j}^{\ell}\circ\mathrm{i}^{\ell}(v_{x}) V_{i j} +R(v_x) 
\end{equation}
where $\| R \|_{W^{k p}}<\frac{\epsilon}{2 \|l \|}$. Note that $l$ to be symmetric  
is equivalent to 
\begin{equation}
\label{eq-6lemma-simetria-l}
l(f_{\alpha}\circ \mathrm{t}^{\ell}\circ \varphi^{\ell}\circ \mathrm{i})=l(f_{\alpha}\circ \mathrm{t}^{\ell}\circ\mathrm{i}^{\ell}) 
\end{equation}
Applying $l$ in the eq. \eqref{eq-5lemma-simetria-l}, and using \eqref{eq-6lemma-simetria-l}
we conclude:
\begin{equation}
\label{eq-7lemma-simetria-l}
l\Big(\int_{\mathrm{s}^{-1}(x)} f_{\alpha}\circ \mu(g,v_x) \, \tilde{\nu}_{x}\Big)=
\sum_{i j} l(f_{\alpha}\circ \mathrm{t}^{\ell}\circ\mathrm{i}^{\ell}) V_{i j} +l(R) 
\end{equation}


By the same argument as in eq.\eqref{eq-3lemma-simetria-l}, and taking
in consideration that 
$f_{\alpha}(v_x)=f_{\alpha}\circ \mathrm{t}^{\ell}\circ\mathrm{i}^{\ell}(v_x)$
we have:
\begin{equation}
\label{eq-8lemma-simetria-l}
f_{\alpha}(v_x) \int_{\mathrm{s}^{-1}(x)}  \tilde{\nu}_{x}=
\sum_{i j} f_{\alpha}\circ \mathrm{t}^{\ell}\circ\mathrm{i}^{\ell}(v_x) V_{i j} +\widehat{R}(v_x) 
\end{equation}
where $\| \widehat{R} \|_{W^{k p}}<\frac{\epsilon}{2 \|l \|}$.
Applying $l$ on the previous equation we have:
\begin{equation}
\label{eq-9lemma-simetria-l}
l\Big(f_{\alpha} \int_{\mathrm{s}^{-1}(x)}  \tilde{\nu}_{x}\Big)=
\sum_{i j} l(f_{\alpha}\circ \mathrm{t}^{\ell}\circ\mathrm{i}^{\ell}) V_{i j} +l(\widehat{R}) 
\end{equation}

Arbitrary of the choise of $\epsilon$ and  eq. 
\eqref{eq-7lemma-simetria-l} and 
\eqref{eq-9lemma-simetria-l} conclude the proof of eq. 
\eqref{eq-0lemma-simetria-l} and hence the 
proof of the lemma. 

\end{proof}


\subsection{Linear version of Principle of Symmetric Criticality of Palais}
\begin{lemma}
\label{MALEX-lemma-palais-l-orbit-like-version}
Let $\mathcal{G}^{\ell}\rightrightarrows E$ be the linear Lie groupoid whose orbits
are the leaves of the orbit-like foliation $\mathcal{F}^{\ell}$ on $E$. 
Let $l$ be a linear functional on $C^{\infty}_{c}(E^{\delta})$. Assume that
\begin{enumerate}
\item[(a)] $l(b)=0$, for all $\mathcal{F}^{\ell}$-basic function $b\in C^{\infty}_{c}(E^{\delta})$
\item[(b)]  $l$ is symmetric with respect to $\mathcal{F}^{\ell}$ i.e., fulfills Definition \ref{MALEX-definition-Av-orbitlike}
\end{enumerate}
Then $l=0.$
\end{lemma}
\begin{proof}
Consider $f\in C^{\infty}_{c}(E^{\delta}).$ 
\begin{eqnarray*}
0 & \stackrel{(*)}{=} & l\big(\mathrm{Av}(f)\big)\\
  & \stackrel{(**)}{=} & l\big(f \big)
\end{eqnarray*}
where (*) follows from item (a) and Lemma \ref{MALEX-lemma-Av-eh-basico} and (**) follows from 
Lemma \ref{MALEX-lemma-Av-l-orbitlike}. The arbitrariness of choice of  $f$ implies that $l=0$.
\end{proof}
Let us illustrate the principle described above.
Let  $J:C^{\infty}_{c}\big(E^{\delta}\big)\to\mathbb{R}$ be the functional:     
$$ J(f)= c_1\int_{E^{\delta}}  \langle \nabla f, \nabla f \rangle \nu + \frac{1}{2}\int_{E^{\delta}} k  f^{2} \nu+ c_2\int_{E^{\delta}}   f^{c_3} \nu $$
where $\langle \cdot, \cdot\rangle=\mathrm{g}^{^{0}}$ is the Sasaki metric, $\nu$ is the volume induced by the Sasaki metric,
 $k:E^{\delta}\to \mathbb{R}$ is a $\mathcal{F}^{\ell}$-smooth function and $c_i$ and $c_3>0$ are constant.    
\begin{lemma}
\label{MALEX-lemma-palais-J-symmetric-linear}
Let $f\in C^{\infty}_{c}\big(E^{\delta}\big)$ be a smooth function and $b\in C^{\infty}\big(E\big)$ a smooth  $\mathcal{F}^{\ell}$ basic function. 
Set 
$$l(f)=dJ (b)(f)=Q(b,f)= 2c_{1}\int_{E^{\delta}}  \langle \nabla b, \nabla f \rangle \nu + \int_{E^{\delta}} k b f \nu + c_{2} c_3\int_{E^{\delta}}   b^{c_3-1}f \nu  ,$$   
Then $l:C^{\infty}_{c}\big(E^{\delta}\big)\to\mathbb{R} $ is symmetric with respect to volume preserves linearized vector fields, 
i.e., fulfills Definition \ref{MALEX-definition-Av-orbitlike}
\end{lemma}
\begin{proof}
\begin{eqnarray*}
Q(b\circ\varphi_t,f\circ\varphi_t) & = & 2 c_{1}\int_{E^{\delta}}  \langle \nabla \big(b\circ\varphi_{t}\big), \nabla \big(f\circ\varphi_{t}\big) \rangle \nu + 
\int_{E^{\delta}} k \big(b\circ\varphi_{t}\big) \big(f\circ\varphi_{t}\big) \nu \\
& +&   c_{2} c_{3}  \int_{E^{\delta}}  \big(b\circ\varphi_{t}\big)^{c_{3}-1} \big(f\circ\varphi_{t}\big) \nu\\
& \stackrel{(I)}{=} & 2 c_{1} \int_{E^{\delta}}  \langle \nabla \big(b\circ\varphi_{t}\big), \nabla \big(f\circ\varphi_{t}\big) \rangle \nu + 
\int_{E^{\delta}} k  b  f +   c_{2} c_{3}   b^{c_3-1}f  \nu\\
& \stackrel{(II)}{=} &  2 c_{1} \int_{E^{\delta}}  \langle \nabla^{E} \big(b\circ\varphi_{t}\big), \nabla^{E} \big(f\circ\varphi_{t}\big) \rangle \nu + 
\int_{E^{\delta}} k  b  f +   c_{2} c_{3}   b^{c_3-1}f  \nu\\
&  \stackrel{(III)}{=} & 2 c_{1} \int_{E^{\delta}}  \langle \nabla^{E}  b, \nabla^{E} f \rangle_{\varphi_{t}} \nu + 
\int_{E^{\delta}} k  b  f +   c_{2} c_{3}   b^{c_3-1}f  \nu \nu\\
&  \stackrel{(II)}{=} & 2 c_{1} \int_{E^{\delta}}  \langle \nabla  b, \nabla f \rangle_{\varphi_{t}} \nu + 
\int_{E^{\delta}} k  b  f +   c_{2} c_{3}   b^{c_3-1}f  \nu\\
&  \stackrel{(I)}{=} & 2 c_{1} \int_{E^{\delta}}  \langle \nabla  b, \nabla f \rangle \nu + 
\int_{E^{\delta}} k  b  f +   c_{2} c_{3}   b^{c_3-1}f  \nu\\
& = &Q(b,f), 
\end{eqnarray*}
 where $(I)$ follows from the fact that $\varphi_{t}$ preserves $\nu$, $(II)$ from the fact that $\nabla^{E}b=\nabla b$ (where $\nabla^{E}$ is the induced
Riemannian connection on the fibers) and   $(III)$ from the fact that $\varphi_{t}$ induces isometries between the fibres of $E$, 
see Lemma \ref{MALEX-lemma-fluxo-isometiras}. 
\end{proof}
\subsection{A few words about the avarage of a SRF  $\mathcal{F}$  on $E$ }
\label{MALEX-subsection-words-about-general-avarage}
We end this section briefly presenting an avarage operator $\mathrm{Av}_{\mathcal{F}}$ 
for a general SRF $\mathcal{F}$ on $(E,\langle\cdot, \cdot \rangle)$  where $\langle\cdot, \cdot \rangle$ is a  Sasaki metric.
Although we can not (at least until now)
use this operator to have the Palais principle, 
that will be  a fundamental ingredient to construct  the constant 
scalar curvature, the operator  $\mathrm{Av}_{\mathcal{F}}$  will be used  to define  $\mathcal{F}$-Sobolev spaces. 
 
Since in this section, we are dealing with the two foliations $\mathcal{F}^{\ell}\subset\mathcal{F},$ we need two different notations for basic 
functions with respect to these  foliations. Let $C^{\infty}_{c}(E^{\delta})_{b}$ denote (as usual)  the space of basic functions 
with respect to $\mathcal{F}^{\ell}$
and $ C^{\infty}_{c}(E^{\delta})^{\mathcal{F}}$  denote  basic functions with respect to  $\mathcal{F}.$ 
As we saw in previous sections $\mathrm{Av}:C^{\infty}_{c}(E^{\delta})\to C^{\infty}_{c}(E^{\delta})_{b}.$
Note that for a fixed $x_0\in B$  the foliation $\mathcal{F}^{\ell}$ intersects the fiber $E_{x_{0}}.$ This allow us  to project
functions of  $C^{\infty}_{c}(E^{\delta})_{b}$ into functions of $C^{\infty}_{c}(E^{\delta}_{x_0})_{b}$ 
 i.e.,  into the space of basic functions of $(\mathcal{F}_{x_{0}}^{\ell}, E_{x_{0}})$ with compact support on the fiber $E^{\delta}_{x_0}.$
 Hence we can define the restriction operator 
$\mathcal{R}^{\ell}_{x_0}:C^{\infty}_{c}(E^{\delta})_{b}\to C^{\infty}_{c}(E^{\delta}_{x_0})_{b}.$
Since $E_{x_{0}}$ is a vector space, we  have the Mendes-Radeschi avarage on this space \cite[Lemma 21]{mendes2018slice}, i.e, 
$\mathrm{Av}_{\mathcal{F}_{x_0}}:  C^{\infty}_{c}(E^{\delta}_{x_0})_{b} \to C^{\infty}_{c}(E^{\delta}_{x_0})^{\mathcal{F}}.$ Note that
two different leaves  of $\mathcal{F}_{x_{0}}$ (let us call them plaques)
 may belong to the same leaf $L\in\mathcal{F}$, but in this case they are isometric to each other, because the linearized holonomy 
sends one plaque to other isometrically. This observation allow us to extend $\mathcal{F}_{x_{0}}$-basic function on $E_{x_{0}}$ to
$\mathcal{F}$-basic function  on $E$, i.e.,to define the   operator
$\mathcal{E}:  C^{\infty}_{c}(E^{\delta}_{x_0})^{\mathcal{F}}\to  C^{\infty}_{c}(E^{\delta})^{\mathcal{F}}.$ Finally we  define
the desired operator: $ \mathrm{Av}_{\mathcal{F}}:C^{\infty}_{c}(E^{\delta})\to C^{\infty}_{c}(E^{\delta})^{\mathcal{F}}$ as 
$ \mathrm{Av}_{\mathcal{F}}=  \mathcal{E}\circ \mathrm{Av}_{\mathcal{F}_{x_0}}\circ\mathcal{R}^{\ell}_{x_0}\circ\mathrm{Av}.$



¨


\section{Principle of Symmetric Criticality of Palais on $M$}
\label{MALEX-section-principle-symmetry-palais-M}

We start by presenting a metric constructed with Sasaki metrics on tubular neighborhoods and partition of unity, 
that is analogous to the proof of \cite[Proposition 1.29]{morita2001geometry}. 

\begin{lemma}
\label{MALEX-lemma-particao-unidade-F-invariante}
Consider an open cover of $\mathcal{F}$-tubular neighborhoods $\{\mathrm{tub}_{r_{\alpha}}(L_{q_{\alpha}})\}$ of $M$. Then
\begin{enumerate}
\item[(a)] there exists a locally finite 
open cover of tubular neighborhoods $\mathrm{tub}_{\delta_{i}}(L_{p_{i}})$ that is a refinement of  $\{\mathrm{tub}_{r_{\alpha}}(L_{q_{\alpha}})\}.$ 
In addition $\mathrm{tub}_{\delta_{i}/3}(L_{p_{i}})$ is still an open covering of $M$. 
\item[(b)] There exists a $\mathcal{F}$-partition of unity $\{\rho_{i}\}$ 
with support on $\mathrm{tub}_{\delta_{i}}(L_{p_{i}})$ so that $\rho_{i}=1$ when restrict to $\mathrm{tub}_{\delta_{i}/3}(L_{p_{i}})$. 
The partition of unity $\{\rho_{i}\}$  is said to be subordinate to $\{\mathrm{tub}_{r_{\alpha}}(L_{q_{\alpha}})\}.$ 
\end{enumerate}
\end{lemma}
Given an open cover of tubular neighborhoods $\{\mathrm{tub}_{r_{\alpha}}(L_{q_{\alpha}})\}$ and a subordinate partition of unity $\{\rho_{i}\}$, 
we can define, via normal exponential map,  a Sasaki metric $\langle \cdot, \cdot\rangle_{i}$ on $\mathrm{tub}_{\delta_{i}}(L_{p_{i}})$ 
(see notation in Lemma  \ref{MALEX-lemma-particao-unidade-F-invariante}) and a $\mathcal{F}$-basic metric on $M$ as: 
\begin{equation}
\label{MALEX-eq-metrica-basica-J-simetrica}
 \langle \cdot, \cdot \rangle=\sum_{i} \rho_{i} \langle\cdot,\cdot\rangle_{i}  
\end{equation}
\begin{proposition}
\label{MALEX-proposition-simetria-J-M}
Let $(M,\mathcal{F})$ be an orbit-like foliation and 
$k:M\to \mathbb{R}$ be a $\mathcal{F}$-smooth function. 
 Let $\langle \cdot, \cdot, \rangle$ be the metric construct above.
Set $$l(f)=dJ_{g}(b)(f)=Q(b,f)= 2 c_{1}\int_{M}  \langle \nabla b, \nabla f \rangle \nu + \int_{M} k b f \nu + c_{2} c_{3}  \int_{M} b^{c_3-1}f  \nu,$$   
Assume that for a given   smooth   $\mathcal{F}$-basic function   $b\in C^{\infty}\big(M\big)$ 
we have that $l(\tilde{b})=d J_{g}(b)(\tilde{b})=0$ for all basic functions $\tilde{b}$. Then
$d J(b)=0$
\end{proposition}
\begin{proof}

Let $\{\tilde{\rho}_{j}\}$ be a partition of unity subordinate to  $\{\mathrm{tub}_{\delta_{i}/3}(L_{p_{i}})\}.$
Given an $f$, by setting $f_{j}=\tilde{\rho}_{j}f$ we have that $f=\sum_{j} f_{j}$. In order
to check that $l(f)=0$ it suffices to check that $l(f_{j})=0$. For a fixed $j$ there exists $i$ so that
$\mathrm{supp}(f_{j})\subset U_{i} :=\mathrm{tub}_{\delta_{i}/3}(L_{p_{i}})$. By construction $\langle \cdot, \cdot\rangle$
restrict to $ \mathrm{tub}_{\delta_{i}/3}(L_{p_{i}})$ coincides with the Sasaki metric $\langle\cdot,\cdot\rangle_i.$ Therefore:
\begin{eqnarray*}
l(f_{j})=dJ_{g}(b)(f_{j}) &= & 2 c_{1}\int_{M}  \langle \nabla b, \nabla f_{j} \rangle \nu + \int_{M} k b f_{j}\nu + c_{2} c_{3}  \int_{M} b^{c_3-1}f_{j} \nu\\
& = & \int_{U_{i} } \Big( 2 c_{1}\langle \nabla b, \nabla f_{j} \rangle_{i} +  k b f_{j} +  c_{2} c_{3}   b^{c_3-1}f_{j}\Big) \nu \\
& = & l_{i}(f_{j})
\end{eqnarray*}

From Lemma \ref{MALEX-lemma-palais-J-symmetric-linear} $l_{i}:C^{\infty}_{c}(U_{i})\to \mathbb{R}$ is symmetric. 
Since, by hypothesis  $l(\tilde{b})=0$ for all basic functions, it is also true that $l_{i}(\tilde{b})=0$  $,\forall \tilde{b}\in C^{\infty}_{c}(U_{i}).$
From Lemma \ref{MALEX-lemma-palais-l-orbit-like-version} we infer that $l_{i}=0.$ Hence $l(f_{j})=0$ and this finishes the proof.  
\end{proof}


\section{Sobolev spaces of basic functions on SRF spaces}
\label{MALEX-sec-sobolev-spaces}

Let us start by  extending the operator defined in Section \ref{MALEX-subsection-words-about-general-avarage} 
to the manifold $M$ using  partition of unity. More precisely, given a covering  
of tubular neighborhoods  $\{\mathrm{tub}_{r_{\alpha}}(L_{q_{\alpha}})\},$
we have a  partition of unity $\{\rho_{i}\}$ subordinate to it; see Lemma \ref{MALEX-lemma-particao-unidade-F-invariante} for notations.
We  define $\mathrm{Av}_{i}:C^{\infty}(U_{i})\to C^{\infty}(U_{i})^{\mathcal{F}}$ 
as the operator  $\mathrm{Av}_{\mathcal{F}}$ defined on $E^{\delta_{i}}$once we have identified,
via normal exponential map,  the tubular neighborhood  $U_{i}=\mathrm{tub}_{\delta_{i}}(L_{p_{i}})$ 
 with normal space $E^{\delta_{i}}=\nu^{\delta_{i}}(L_{p_{i}})$ of $L_{p_{i}}$.
 We finally define $ \mathrm{Av}_{\mathcal{F}}:C^{\infty}(M)\to C^{\infty}(M)^{\mathcal{F}}$ as: 
\begin{equation}
\label{MALEX-eq-Av-SRF}
\mathrm{Av}_{\mathcal{F}}(f):= \sum_{i} \mathrm{Av}_{i}(\rho_{i} f).
\end{equation}
Following  \cite[Lemma 21]{mendes2018slice}  it is possible to check that $\mathrm{Av}_{i}:C^{\infty}(U_{i})\to C^{\infty}(U_{i})^{\mathcal{F}}$ is a continuous 
operator with respect to the Sobolev metrics. 

Once established the average operator 
 we proceed by formalizing the construction of the \textit{Sobolev Space of basic distributions} and prove some the corresponding Kondrakov Theorem \ref{thm:kondrakov} in the context. 

Let $(M,\ga)$ be a Riemannian manifold, $\mathcal{F}$ be a SRF on $(M,\ga)$ and denote by $\mu_g$ the Radon measure on $M$ induced by $\ga$. Fix a nonnegative integer $k$ and $1\leq p < \infty$. 

We define
\[
	C^{k,p}(M,\ga)
	=
	\set{
		u \in C^\infty(M):
		\|u\|_{k,p,\ga}
		:=
		\left[
		\sum_{j=0}^k
		\int_M
			|\nabla^j u|^p
		\dif (\mu_{\ga})
		\right]^{1/p}
		<\infty
	},
\]
where $\|\cdot\|_{k,p,\ga}\colon C^{k,p}(M,\ga) \to [0,\infty[$ is called the $(k,p)$-\textit{Sobolev norm} on $(M,\ga)$. In this context, we can define the \textit{Sobolev Space} of $\mathcal{F}-$basic distributions on $(M,\ga)$:

\begin{definition}
The Sobolev space $W^{k,p}(M,\ga)^\mathcal{F}$ of $\mathcal{F}$-basic distributions in $M$ is 
 the Banach space of distributions in $W^{k,p}(M,\ga)$ that are
constant along the leaves of the Singular Riemannian Foliation $\mathcal{F}$.
\end{definition}
\begin{remark}
For each $u\in W^{k,p}(M,\ga)^\mathcal{F} $ we can find a sequence of smooth $\mathcal{F}$ basic functions $\{u_n\}$ that 
converges (with respect to the Sobolev norm) to $u$. In fact consider a sequence  of smooth functions $\{\tilde{u}_{n}\}$ that converges
to $u$. Then $u_{n}=\mathrm{Av}_{\mathcal{F}}(\tilde{u}_{n})$ converges to $u=\mathrm{Av}_{\mathcal{F}}(u)$. 
\end{remark}

This definition is compatible with the usual definition of $W_{\ga}^{k,p}(M)$ (see \cite[Def. 2.1, p.21]{Hebey_2000}) in the sense that the arrows in the commutative diagram below correspond to linear $\|\cdot\|_{k,p,\ga}$-continuous embeddings:
\[
	\begin{tikzcd}
		{W^{k,p}(M,\ga)^\mathcal{F}} \arrow[d,symbol=\subset] &C^{k,p}(M,g)^\mathcal{F} \arrow[hook', l] \arrow[d,symbol=\subset]
		\\
		{W^{k,p}(M,\ga)} &C^{k,p}(M,\ga) \arrow[hook', l]
\end{tikzcd}.
\]

\begin{remark}\label{	}
Suppose that $M$ is a compact manifold. Any two Riemannian metrics on $M$ yield equivalent $(k,p)$-Sobolev norms (see \cite[Prop. 2.2, p.22]{Hebey_2000}). For such reason, we omit mention to Riemannian metrics since we only consider compact manifolds. Moreover, $W^{k,p}(M)$ can be equivalently defined as the completion of $C^\infty(M)$ with respect to any $(k,p)$-Sobolev norm.
\end{remark}

It is classical nowadays that when considering manifolds with isometric actions, assuming that each orbit has infinity cardinality leads to better compact embeddings coming from Kondrakov's Theorem (see \cite[Theorem 9.1, p.252]{Hebey_2000}). More recently, \cite{cavenaghi2021} it was observed that when the group acts properly it suffices indeed that there exist at least on orbit of positive dimension to guarantee better compact embeddings, leading, for instance, to new proofs for the Yamabe problem in this scenario, but also to analogous results to Kazdan--Warner in the $G$-invariant setting.

Here we obtain a Kondrakov-type theorem for SRF with compact leaves on compact manifolds:

\begin{theorem}[Kondrakov-type theorem]\label{thm:kondrakov}
Suppose that $M$ is a connected compact Riemannian manifold and $\mathcal{F}$ is a SRF on $M$ whose leaves are closed. 
Then there exists $p_0 > p^* := np/(n-p)$ such that given $1<q<p_0$, the canonical embedding $W^{1,p}(M)^\mathcal{F} \hookrightarrow L^q(M)$ is compact.
\end{theorem}

Before we proceed  to the proof, recall that the regular stratum  $M^{\rm{reg}}$ of $\mathcal{F}$ on $M$ is open and dense set in $M$,
see Section \ref{MALEX-Section-SRF}.
Let $k:=\dim L_{x}$ for an $x \in M^{\rm{reg}}$. Using Theorem \ref{MALEX-theorem-desingularizacao}, it is possible to check that 
 there exist trivializing coordinate charts $\{(\Omega,\varphi)\}$ on $M^{\rm{reg}}$ with properties
\begin{enumerate}[(i)]
    \item $\varphi(\Omega) = U\times V\subset \mathbb{R}^k\times \mathbb{R}^{n-k};$
    \item $\forall y \in \Omega,~U\times \mathrm{pr}_2(\varphi(y)) \subset \varphi(Gy\cap \Omega),$ where $\mathrm{pr}_2 : \mathbb{R}^{k}\times \mathbb{R}^{n-k} \to \mathbb{R}^{n-k}$ is the second projection. 
\end{enumerate}
Property (ii) implies that if $f : M \to \mathbb{R}$ is a $\mathcal{F}$-invariant function (hence constant along the leaves), then $f \circ\varphi^{-1}$ is constant on its first coordinate.

\begin{proof}[Proof of Theorem \ref{thm:kondrakov}]
Let $\mu$ be a Radon measure on $M$ induced by any Riemannian metric on $M$. Since $\mu\left(M\setminus M^{\mathrm{reg}}\right) = 0$, then
\[
	\int_M |f|^q\dif\mu
	=
	\int_{M^{\rm{reg}}\cup \left(M\setminus M^{\rm{reg}}\right)}|f|^q\dif\mu
	=
	\int_{M^{\rm{reg}}}|f|^q\dif\mu
\]
for any $1\leq q<\infty$, $f \in L^q(M)=W^{0,q}(M)$.
Since there is a leaf of dimension at least $1$, if $d^*  $ denotes the dimension of a principal orbit, then $d^* \geq 1.$ Cover $M^{\rm{reg}}$ with finitely many trivializing charts $\{(\Omega,\varphi)\}$ such that $\varphi(\Omega) = U\times V \subset \mathbb{R}^{d^*}\times \mathbb{R}^{n-d^*}$ and $U\ni x \mapsto f\circ\varphi^{-1}(x,\cdot)$ is constant.

By the Fubini Theorem on $U\times V$ and the Sobolev Embedding Theorem for open sets in $\mathbb{R}^{n-d^*}$, we conclude that there are constants $C,K > 0$ such that given $1 \leq q \leq p(n-d)/(n-d-p)$, it holds that
\[
	\left(
		\int_{\Omega}|f|^q\dif\mu
	\right)^{1/q}
	=
	\left(
		\int_{U\times V}|f\circ\varphi^{-1}|^q(\varphi^{-1})^{\ast}\left(\dif\mu\right)
	\right)^{1/q},
\]
\begin{multline*}
    \left(\int_{U\times V}\left|f\circ\varphi^{-1}\right|^q(\varphi^{-1})^{\ast}\left(\dif\mu\right)\right)^{p/q} \leq
    \\
    \leq C\int_{V}\left(\left|f\circ\varphi^{-1}\right|^p + \left|\nabla (f\circ\varphi^{-1})\right|^p\right)(\mathrm{pr}_2\circ \varphi^{-1})^{\ast}(\dif\mu) =
    \\
    = K\int_{\Omega}\left(\left|f\right|^p +\left|\nabla f\right|^p\right)\dif\mu.
\end{multline*}

Since there are finitely many open sets on the trivialization atlas, it follows that there exists $C>0$ such that
\begin{align*}
   \|f\|_q =   \left(\int_{M^{\rm{reg}}}|f|^q\dif\mu\right)^{1/q} \leq C\|f\|_{1,p}.
\end{align*}
Therefore, the canonical inclusion $W^{1,p}(M)^\mathcal{F} \hookrightarrow L(M)$ is continuous whenever $1\leq q \leq p(n-d)/(n-d)-p$.

If $\epsilon > 0$ is sufficiently small, then the classic \textit{Kondrakov Theorem} implies that this inclusion is compact whenever $1 \leq q \leq p(n-d+\epsilon)/(n-d+\epsilon - p)$. The continuous function $[1,\infty[\setminus\set{p} \ni t \mapsto pt/(t-p)$ is decreasing, so $p(n-d+\epsilon)/(n-d+\epsilon - p) > p^* = np/(n-p)$.
\end{proof}


\section{The Yamabe problem on manifolds with orbit-like foliation and bundles}
\label{MALEX-section-yamabe-problem}

	In this section we explore the developed machinery for, as proof of concept, study the Yamabe problem \cite{yamabeoriginal,trudinger,schoen-yau1,schoen-yau2,schoen-yau3} in the setting of both orbit-like foliation and fiber bundles, aiming prescribing constant scalar curvature metrics adapted to the induced foliation on both cases. In a near future we shall make it public the respective results on the Kazdan--Warner \cite{kazdaninventiones,kazadanannals,kazdan1975} in this scenario.
	
	\subsection{The proof strategy and its self motivation}
	\label{sec:overview}
	To proceed, fix a closed Riemannian manifold $(M,\ga)$ with an orbit-like foliation $\cal F$ of closed leaves. Also assume that the scalar curvature of $\ga$ is basic. We start making it clear that given some constant $c$, on both Theorems \ref{thm:orbitlikeintro} and \ref{thm:bundlesintro}, to search for a smooth basic and positive function $u : M \to \mathbb{R}$ such that $\mathrm{scal}_{u^{\frac{4}{n-2}}g} = c$ is equivalent to solve the following elliptic PDE:
		\begin{equation}\label{eq:PDE}
			4b_n\Delta_{\ga}u - \mathrm{scal}_{\ga}u + cu^{\gamma_n} = 0,
		\end{equation}
		where $c > 0$ and $\gamma_n := \dfrac{n+2}{n-2}$ and $b_n :=\dfrac{n-1}{n-2}$. 
		
		To begin with, it suffices to assume that $\mathrm{scal}_{\ga}$ is only continuous. This manner, we proceed by considering the following functional
		\begin{equation*}
			J(u) = 2b_n\int_M|\nabla u|_{\ga}^2 + \frac{1}{2}\int_M\mathrm{scal}_{\ga}u^2 - \frac{c}{2^*}\int_Mu^{2^*}
		\end{equation*}
		a priori defined in the Sobolev space $W^{1,2}(M)$. To weakly solve the PDE \eqref{eq:PDE} consists of  finding a critical point $u$ of $J$, namely, $dJ(u)(v) = 0 \forall v \in W^{1,2}(M)$. For both the proofs of Theorems \ref{thm:orbitlikeintro} and \ref{thm:bundlesintro} these critical points correspond to local minima.
		
		Note however that, since $J$ is a basic functional in the sense of Definition \ref{MALEX-definition-Av-orbitlike}, it is possible to restrict $J$ to $W^{1,2}(M)^{\cal F}$. On the one hand, to find a critical point for $J$ restricted to $W^{1,2}(M)^{\cal F}$ means to exist $u\in W^{1,2}(M)^{\cal F}$ such that $dJ(u)(v) = 0 \forall v \in W^{1,2}(M)^{\cal F}$. It is in this very point that a result such as the principle of symmetric criticality is needed. Observe that under this former hypothesis, Proposition \ref{MALEX-proposition-simetria-J-M} implies that $dJ(u)(v) = 0\forall v \in W^{1,2}(M)$, meaning that this \textit{basic} critical point is a critical point indeed. 
		
		On the other hand, we justify our interests on finding a basic solution to ensure that \emph{the metric with constant scalar curvature is basic}, hence preserving the foliation geometric structure. It turns out however that such a restriction plays a huge role in the argumentation of finding a minimum for $J$ via variational methods due to the existence of better compactness results such as Theorem \ref{thm:kondrakov}. Note for instance that $\gamma_n$ is a critical exponent for the classical Kondrakov Theorem, i.e, it is such that it is not necessarily true to exist a compact embedding of $W^{1,2}(M)$ in $L^{\gamma_n + 1}(M)$.
		So we naturally proceed to find a basic critical point. 
		
		To do so, we shall look for a local minima for $J$ with specific constraints. For both Theorems \ref{thm:orbitlikeintro} and \ref{thm:bundlesintro} we restrict the analysis to a codimension 1 submanifold $\mathbf{M}_{\cal F}$ of $W^{1,2}(M)^{\cal F}$: 
		\[\mathbf{M}_{\cal F} := \left\{ u \in W^{1,2}( M)^{\cal F} : C\geq u \geq 0 ~\text{a.e}~, \frac{c}{2^*}\int_{ M}u^{2^*} = \epsilon\right\}.\]
		
		A routine argument in variational methods, trivial in this scenario given the Kondrakov type result, then ensure the existence of such a minimum point $u$ for both cases. Since the constraint is a submanifold, we observe that the Lagrange Multiplier equation obtained for this problem can be reduced to the original one, $dJ(u)(v) = 0 \forall v \in W^{1,2}(M)^{\cal F}$ after scaling the original metric, concluding that $u$ is then a basic critical point.
		
		Finally, due to the regularity theory of elliptic PDE's one concludes that the solution $u$ is smooth as long as $\mathrm{scal}_{\ga}$ and $f$ are smooth. In fact, note that the non-linearity on the PDE \eqref{eq:PDE} corresponds to the term $u^{\gamma_n}.$ Since the function
	$F : x \to x^{\gamma_n}$ is of class $C^1$ and the solution $u$ has finite essential supremum, then $F(u) \in W^{1,2}(M)$. An iterative application of Theorem 3.58 in \cite[p. 87]{aubinbook} implies the result. The maximum principle \cite[Proposition 3.75, p.98]{aubinbook} implies that the obtained solution is positive.

\subsection{The Yamabe problem on manifolds with orbit-like foliation}

	We now use the developed machinery to prove Theorem \ref{thm:orbitlikeintro}. We restate it here for convenience:
		
			\begin{theorem}\label{thm:orbitlike}
		Let $ M^n,~n\geq 3$, be a closed Riemannian manifold endowed with a orbit-like foliation $\cal F$ with closed leaves. Then $ M$ has a Riemannian metric of constant scalar curvature for which $\cal F$ is a orbit-like foliation.
	\end{theorem}
	\begin{proof}
 Equip $M$ with the basic metric $\ga=\langle \cdot, \cdot \rangle $ defined in Eq.~\eqref{MALEX-eq-metrica-basica-J-simetrica}.   
We claim that there is $c\geq 0$ such that for any $c' \geq c$ there exists a Riemannian metric $\tilde \ga$ with basic scalar curvature such that 
$\mathrm{scal}_{\tilde \ga} = - c'$. As we already pointed out in Section \ref{sec:overview}, such a $\tilde \ga$ comes from a conformal change. 
Also taking in count this section, we proceed finding a basic critical point, what shall finishes the proof. 
	
	Take $c\geq 0$ such that
  \[\left(\frac{2^*}{2}\right)\min_{M}\mathrm{scal}_{\ga}\mathrm{vol}(M)^{1-2^*/2} +c\geq  0,\]
  and consider the functional
  \[J(u) =  2b_n \int_M |\nabla u|_{\ga}^2 + \dfrac{1}{2}\int_M\mathrm{scal}_{\ga}u^2 + \dfrac{c}{2^*}\int_Mu^{2^*}\]
  defined in $W^{1,2}(M)^{\cal F}$. Let us show that $J$ is coercive.
  
		 To do so, note that the H\"older inequality implies that
			\begin{align}
				\left(\int_M u^2\right) &\leq \mathrm{vol}(M)^{1-2/2^*}\left(\int_M u^{2^*}\right)^{2/{2^*}}. \label{eq:inequality1.5}
			\end{align}
		
		We now consider two separate cases depending if $\min \mathrm{scal}_{\ga}\leq 0$ or $\min \mathrm{scal}_{\ga}> 0$. Respectively we have 
			\begin{align*}
				J(u) &\geq 2b_n \int_M |\nabla u|_{\ga}^2 + \dfrac{1}{2}\min_M \mathrm{scal}_{\ga} \mathrm{vol}(M)^{1-2/2^*}\left(\int_M u^{2^*}\right)^{2/{2^*}} + \dfrac{c}{2^*}\int_Mu^{2^*};\\
		 J(u) &\geq 2b_n \int_M |\nabla u|_{\ga}^2 + \dfrac{c}{2^*}\int_Mu^{2^*}.
			\end{align*}
	We pass to the submanifold
	\[\mathbf{M}_{\cal F} := \left\{ u \in W^{1,2}(\overline M)^{\cal F} : C\geq u \geq 0 ~\text{a.e}~, \frac{c}{2^*}\int_{\overline M}u^{2^*} = \epsilon\right\},\]
	where we have liberty on the choice of  $\epsilon$.
	We then check that $J\Big|_{\mathbf{M}_{\cal F} }$ is coercive and weakly lower semiconinutous. To do so, observe that considering the imposed restrictions one has
			\begin{align}
				\label{eq:caso1}J(u) &\geq 2b_n \int_M |\nabla u|_{\ga}^2 + \dfrac{1}{2}\min_M \mathrm{scal}_{\ga} \mathrm{vol}(M)^{1-2/2^*}\left(\frac{2^*\epsilon}{c}\right)^{2/{2^*}} + \epsilon;\\
			\label{eq:caso2}J(u) &\geq 2b_n \int_M |\nabla u|_{\ga}^2 + \epsilon.
			\end{align}
			from where it follows that $J\Big|_{\mathbf{M}_{\cal F} }$ is coercive.
			
			It is also immediate to conclude that $\mathbf{M}_{\cal F}$ is weakly closed. Indeed, take $\mathbf{M}_{\cal F}\supset \{u_m\} \rightharpoonup u \in W^{1,2}(M).$ Once $\{u_m\} \subset W^{1,2}(M)^{\cal F}$ and this is a Banach space one has that $u\in W^{1,2}(M)^{\cal F}.$ According to Theorem \ref{thm:kondrakov} one has the compact embedding of $W^{1,2}(M)^{\cal F}$ in $L^{2^*}(M)$, from where it follows that $c\int_M u^{2^*} = \epsilon 2^*.$ Moreover, the sequence $\{u_m\}$ has a pointwise convergent subsequence, so that $C\geq u\geq 0$ almost everywhere. Therefore, $u\in \mathbf{M}_{\cal F}$ and $\mathbf{M}_{\cal F}$ is weakly closed.
			
			 As a last step we observe that $J\Big|_{\mathbf{M}_{\cal F} }$ is weakly lower semicontinuous since: due to Theorem \ref{thm:kondrakov}, any weakly converging sequence $\{u_m\}\subset \mathbf{M}_{\cal F}$ strongly converges in $L^p(M)\cap  \mathbf{M}_{\cal F}$ for every $p\in [1,2^*]$; moreover, once the weak convergence implies that
			\[\liminf_{m\to\infty}\|u_m\|_{1,2} \geq \|u\|_{1,2},\]
        and since $\int_M u_m^{2} \to \int_M u^{2}$ is a convergent sequence, it holds that
			\begin{equation*}
			\liminf_{m\to\infty}\|u_m\|^2_{1,2} =\int_Mu^2+ \liminf_{m\to\infty}\int_M|\nabla u_m|^2,
			\end{equation*}
			and hence
			\begin{equation*}
				\liminf_{m\to \infty} \int_M|\nabla u_m|^2 \geq \int_M|\nabla u|^2.
			\end{equation*}
		Finally, once $\mathrm{scal}_{\ga}$ is continuous and  $u\in \mathbf{M}_{\cal F}$ one concludes that
		\begin{gather*}
			\liminf_{m\to\infty}J(u_m) \geq 2b_n\int_{ M}|\nabla u|^2 +\frac{1}{2}\int_{ M}\mathrm{scal}_{\ga}u^2-\epsilon = J(u).
		\end{gather*}
		
			It then follows that the restriction $J\Big|_{\mathbf{M}_{\cal F}}$ has a minimum $u\in\mathbf{M}_{\cal F}$. Since we have obtained a critical pointed subjected to the artificial constraint $\mathbf{M}_{\cal F}$, we must proceed by looking to the corresponding Lagrange Multiplier associated to this problem. To do so, note now that if $v\in W^{1,2}(M)^{\cal F},$ the Lagrange Multiplier Theorem states that there is $\lambda \in \mathbb{R}$ such that
			\[J'(u)(v) = 4b_n\int_M \langle \nabla u,\nabla v\rangle + \int_{\overline M}\mathrm{scal}_{\ga}uv -c\int_{\overline M} u^{\gamma_n}v = \lambda c\int_{\overline M} u^{\gamma_n}v = H'(u)(v),\]
			where $H^{-1}(0) = \mathbf{M}_{\cal F}\setminus \partial \mathbf{M}_{\cal F}$. We reinforce that $u$ does not lie in $\partial \mathbf{M}_{\cal F}$ since both: we can assume that $u$ is not constant equal to $C$, otherwise it would imply that the original metric already has constant scalar curvature; moreover, the integral constraint avoid $u$ to be identically zero.
						
			We thus conclude that $u$ is a weak solution of \eqref{eq:PDE} with $c$ replaced by $c'=(1+\lambda)c$. On the other hand, by computing $J'(u)(u)$, we conclude that  $1+\lambda > 0$, thus $c'>0$. We then proceed with a rescaling of the resulting metric to obtain the right constant scalar curvature metric. Note that such a scaling is possible given the liberty in the choice of $\epsilon$ in the definiton of $\mathbf{M}_{\cal F}$.
	\end{proof}

	\subsection{Prescribing Riemannian submersion metrics on fiber bundles}
		\label{sec:bundles}

 In this section we discuss the Yamabe problem to the setting of fiber bundles whose fibers are homogeneous spaces. More precisely, we prove:

	\begin{theorem}\label{thm:bundles}
		Let $ M^n,~n\geq 3$, be a closed Riemannian manifold endowed with a Foliation $\cal F$ induced by a fiber bundle such that: 
		\begin{enumerate}[$(i)$]
		    \item The structure group $G$ is compact and has non-abelian Lie algebra;
		    \item The fiber $L$ is an homogeneous space.
		\end{enumerate}
		 Then $ M$ has a Riemannian metric of positive constant scalar curvature for which $\cal F$ is Riemannian.
	\end{theorem}
	
	To unify the discussion in terms of Riemannian foliations, let $\pi : L\hookrightarrow (M,\ga) \to (B,h)$ be a Riemannian fiber bundle with 
	compact structure group $G$ and total space $ M$ closed. In this scenario, the decomposition of $ M$ with respect to the Riemannian foliation induced by the fibers $\{L_x\} = \cal F$ is an example of SRF which leaves are diffeomorphic.
	
	Note that $\pi$ can always be obtained as an associated bundle construction for some principal bundle $P$: there is a $G$-manifold $P$ 
	for which the corresponding $G$-action is free and such that $M = P\times_GL$. We consider the particular case where the $G$-action on $L$ is transitive. 
	This way, $L$ can be seen as the homogeneous space $G/G_l,$ where $G_l$ is the isotropy subgroup at some $l \in L$. Hence, $L$ coincides with the orbit of $G$ through $l$.
	

	\begin{proof}[Proof of Theorem \ref{thm:bundles}]
	Given any $G$-invariant Riemannian metric $\ga_L$ in the fiber $L$, according to \cite[Theorem F]{cavenaghi2018positive} it follows that $\ga_L$ develops positive scalar curvature after a finite \textit{Cheeger deformation} (see \cite{ziller} for more on such deformations). This manner, for simplicity we assume that $\ga_L$ itself has positive scalar curvature. 
	
	Now consider on $ M$ the unique Riemannian submersion metric $\ga$ whose fibers are totally geodesic and isometric to $(L,\ga_L)$. Since we can shrink sufficiently the fibers by the means of a \textit{Canonical Variation} we shall assume that $\ga$ has positive scalar curvature.
	Since $\mathrm{scal}_{\ga_L}$ is $G$-invariant and $L$ is a homogeneous space it means that $\mathrm{scal}_{\ga_L}$ is constant and so $\ga$ is a metric with basic scalar curvature.

Once more we rely in the discussion presented in section \ref{sec:overview}. More precisely, by the means of a conformal change we shall prescribe $c > 0$ as the scalar curvature of a metric $\widetilde \ga = u^{4/n-2}\ga$. Therefore, without further preliminaries, let $\epsilon > 0$ to be chosen conveniently. We search for a critical point of $J$ in
		\[\mathbf{M}_{\cal F} := \left\{ u \in W^{1,2}(M)^{\cal F} : C\geq u \geq 0 ~\text{a.e}~, \frac{c}{2^*}\int_{ M}u^{2^*} = \epsilon\right\},\]
		where $C \geq \left(\dfrac{\epsilon 2^*}{c\mathrm{vol}_{\ga}( M)}\right)^{\frac{1}{2^*}}$.
		
		As we have already seen in the proof of Theorem \ref{thm:orbitlike} the manifold  $\mathbf{M}_{\cal F}$ is weakly closed and $J|_{{\mathbf{M}_{\cal F}}}$ is weakly lower semicotinuous. It only remains to prove that $J|_{\mathbf{M}_{\cal F}}$ is coercive:
	
			Note that if $u\in \mathbf{M}_{\cal F}$ then
			\begin{align}\label{eq:inequality1}
				J(u) &\geq 2b_n\int_{M}|\nabla u|_{\textsl{g}}^2 + \dfrac{\min_{M}\mathrm{scal}_{\ga}}{2}\int_{ M}u^2 - \dfrac{c}{2^*}\int_{ M}u^{2^*}\\
				&=2b_n\int_{ M}|\nabla u|_{\textsl{g}}^2 + \dfrac{\min_{ M}\mathrm{scal}_{\ga}}{2}\int_{ M}u^2 - \epsilon.
			\end{align}
			Therefore, according to  Poincar\'e inequality, since $\min_{ M} \mathrm{scal}_{\ga}\geq 0$,  $J(u) \to \infty$ if $\|u\|_{W^{1,2}(M)^{\cal F}} \to \infty$.

		The remaining of the argument follows equaly to the ones in the proof of Theorem  \ref{thm:orbitlike}.
	
	As a lest step we argue that we can indeed obtain a Riemannian submersion metric. To do so, first note that since $u$ is a basic function, 
	then the obtained metric $u^{\frac{4}{n-2}}\ga$ makes the foliation $\cal F$ Riemannian. Moreover, since $\ga$ is a Riemannian submersion metric there is a Riemannian metric $\bar \ga$ in $ M/\cal F$ for such that $\pi_{\ast}\ga = \bar \ga$. Hence, since $u$ is basic one concludes that 
	$\pi_{\ast}(u^{\frac{4}{n-2}}\ga) = u^{\frac{4}{n-2}}\pi_{\ast}\ga = u^{\frac{4}{n-2}}\bar \ga := \textsl{h}$.
	
	Finally, the fiber bundle $L \hookrightarrow ( M,\widetilde \ga) \rightarrow ( M/\cal F,\textsl{h})$ with $\widetilde  \ga = u^{\frac{4}{n-2}}\ga$ satisfies the thesis.
		\end{proof}
		
		We finish this section with a simply application of our results to the prescription of constant scalar curvature metrics on exotic spheres.
		
		\subsubsection{Applications to bundles whose total space are exotic spheres}
				Eells and Kuiper in \cite{ek} computed the number of $7$ (respectively $15$)-exotic spheres that are realized as total spaces of sphere bundles. Therefore, by setting $G = O(n+1),~n = 7,~15, F = S^n$, considering this and the discussion in section \ref{sec:bundles}, a simply application of Theorem \ref{thm:bundles} gives:
\begin{theorem}\label{thm:grourigas}
 16 (resp. $4.096$) from the $28$ (resp. $16.256$) diffeomorphisms classes of the 
$7$-dimensional (resp. $15$)-exotic spheres admit metrics of positive constant scalar curvature. Moreover, these can be taken as Riemannian submersion metrics when such spaces are considered as the total space of sphere bundles.
\end{theorem}



\bibliographystyle{alpha}
	\bibliography{main}
\end{document}